\documentclass[]{article}

\usepackage{amsmath}
\usepackage{amssymb}
\usepackage{amsthm}
\usepackage[toc,page]{appendix}
\usepackage{bbold}
\usepackage[justification = centering]{caption}
\usepackage{centernot}
\usepackage{fancyhdr}
\usepackage{mathrsfs}
\usepackage{wrapfig}
\usepackage{lineno}
\usepackage{float}

\usepackage{cite}

\newtheorem{theorem}{Theorem}[section]

\newtheorem{corollary}[theorem]{Corollary}

\newtheorem{lemma}[theorem]{Lemma}

\newtheorem{proposition}[theorem]{Proposition}
\newtheorem{claim}[theorem]{Claim}
\newtheorem{fact}[theorem]{Fact}

\newcommand{\claimproof}{\renewcommand{\qedsymbol}{$\diamond$}}  
\newenvironment{claimpf}{\begin{proof}}{\end{proof}}

\usepackage{graphicx}
\graphicspath{ {./Images/} }

\newcommand{\swhom}[3][\Gamma]{#2\to_{#1}#3}

\newcommand{\swequiv}[3][\Gamma]{#2 \cong_{#1}#3}

\newcommand{\sweq}[1][$\Gamma$]{#1-equivalent}
\newcommand{\swhomom}[1][$\Gamma$]{#1-homomorphism}

\newcommand{\swg}[2][\Gamma]{sw_{#1}(#2)}

\newcommand{\hcol}[1][$H$]{#1-\textsc{Colouring}}
\newcommand{\hhom}[1][$H$]{\textsc{Hom}-#1}

\newcommand{\swhhom}[2][$\Gamma$]{#1-\textsc{Hom}-#2}

\begin{document}
	\title{A dichotomy theorem for $\Gamma$-switchable $H$-colouring on $m$-edge-coloured graphs}
	\author{Richard Brewster\thanks{Department of Mathematics and Statistics, Thompson Rivers University, Kamloops, BC, Canada}, Arnott Kidner\thanks{Institut f\"{u}r Mathematik, Universit\"{a}t Paderborn, Paderborn, Nordrhein-Westfalen, Deutschland}, Gary MacGillivray\thanks{Department of Mathematics and Statistics, University of Victoria, Victoria, BC, Canada}}
	\date{}
	\maketitle
	
	\begin{abstract}
 Let $G$ be a graph in which each edge is assigned one of the colours $1, 2, \ldots, m$, and let $\Gamma$ be a subgroup of $S_m$.
 The operation of switching at a vertex $x$ of $G$ with respect to an
element $\pi$ of $\Gamma$ permutes the colours of the edges incident with $x$ according to $\pi$.  
We investigate the complexity of whether there exists a sequence of switches that transforms a given $m$-edge-coloured graph $G$ so that it has an edge-colour-preserving homomorphism to a fixed $m$-edge-coloured graph $H$ and give a dichotomy theorem in the case that $\Gamma$ acts transitively.
	\end{abstract}

\section{Introduction}\label{sec:intro}

A \emph{homomorphism} of a graph $G$ to a graph $H$ is a mapping $f: V(G) \to V(H)$ such that whenever $uv \in E(G)$ we have $f(u)f(v) \in E(H)$. The vertices of $H$ may be regarded as colours so a homomorphism is an assignment of colours to the vertices of $G$ such that adjacent vertices in $G$ are assigned colours that are adjacent in $H$.  Thus, homomorphisms generalize vertex colourings. This viewpoint allows one to study natural questions about colourings for homomorphisms and their generalizations (see~\cite{hel08} and its references). For example, the \hcol[$k$] problem generalizes as follows.  Let $H$ be a fixed graph.

	\begin{quote}
		\hhom\ \\
		\textbf{Instance:} A graph $G$. \\
		\textbf{Question:} Is there a homomorphism of $G$ to $H$?
	\end{quote}		

The result that \hcol[$k$] is polynomial-time solvable for $k \leq 2$ and NP-complete for $k \geq 3$ generalizes to the Hell-Ne\v{s}et\v{r}il Homomorphism Dichotomy Theorem.

\begin{theorem}\label{thm:HN}
    Let $H$ be a fixed graph. Then the \hhom[$H$] problem is polynomial-time solvable if $H$ is bipartite; otherwise, \hhom\ is NP-complete.
\end{theorem}  
(We note that \hhom\ is also called \hcol\ in the literature but we are opting for the former in this paper as it is more convenient for our setting defined below.)

Homomorphisms are naturally defined on other graph-like objects, e.g. digraphs or $m$-edge-coloured graphs, and thus the computational complexity of \hhom\ may also be studied in these settings. The main result of this work is a structural (or combinatorial) dichotomy theorem for the homomorphism problem in the setting of $m$-edge-coloured graphs endowed with an edge colour switching operation. (Formal definitions are below.)  We show, analogous to the Hell-Ne\v{s}et\v{r}il Dichotomy Theorem, the problem is polynomial-time solvable if $H$ is homomorphically equivalent to $K_1$ or $K_2$; otherwise, the problem is NP-complete.

\subsection{Preliminaries}

The general setting for our work is as follows. Let $G=(V,E)$ be a graph. We obtain a \emph{mixed graph} from $G$ by choosing a subset $A \subseteq E(G)$ and an orientation for each element of $A$, so the resulting structure has both edges and arcs. 
We restrict our attention to the case where $G$ is a simple graph, thus obtaining what could (and should) be called a \emph{simple} mixed graph but we drop the qualifier ``simple'' because the context is clear throughout the paper.
Mixed graphs were introduced by Ne\v{s}et\v{r}il and Raspaud~\cite{nes00} in an attempt to unify the homomorphism results for oriented graphs and 2-edge-coloured graphs. For fixed non-negative integers $m, n$, an $(m,n)$\emph{-mixed graph} is a mixed graph $G$ where each edge has been assigned a colour from $\{ 1, 2, \dots, m \}$ and each arc has been assigned a colour from $\{ m+1, m+2, \dots, m+n \}$. 

We now introduce the operation of \emph{switching}. The study of switching 
involves some basic concepts from algebra, for which we refer to the text by Gallian~\cite{Gallian_2017}.
Given an $(m,n)$-mixed graph and a vertex $v$, the operation of \emph{switching at $v$} permutes the colours of edges, the colours of arcs, and the direction of arcs incident with $v$ with respect to some element of a permutation group acting on the edge colours, the arc colours, and the arc directions. 
(Homomorphisms of $(m,n)$-mixed graphs with a switching operation are studied in~\cite{kid21, bre_kid_mac21}, see also~\cite{sen22}.)

While the set-up of this work applies to general $(m,n)$-mixed graphs, establishing a structural dichotomy theorem with a mixture of edges and arcs will likely be complicated as described in Subsection~\ref{subsec:CSP}. (A reduction of the arcs only case to the edges only case for bipartite graphs appears in~\cite{kid21}.) Thus, for this paper we will restricted our attention to simple undirected graphs with coloured edges. 

Formally, for a fixed integer $m \geq 1$, an \emph{$m$-edge-coloured graph} is an ordered pair $G = (G_u, \sigma)$ where $G_u$ is a graph and $\sigma$ is a function that assigns each edge of $G_u$ a colour from $\{ 1, 2, \dots, m \}$. We call $G_u$ the \emph{underlying graph} of $G$. As noted above, we restrict our attention to the case
where $G_u$ is simple. We say $G$ is a \emph{cycle} if the underlying graph $G_u$ is a cycle, and similarly for paths, trees, etc.

To define switching on an $m$-edge-coloured graph,
let $\Gamma \subseteq S_m$ be a permutation group acting on $\{ 1, 2, \dots, m \}$.  Let $G$ be an $m$-edge-coloured graph.  Given a permutation $\pi \in \Gamma$ and a vertex $v \in V(G)$, the $m$-edge-coloured graph $G^{(v,\pi)}$ is obtained by applying $\pi$ to each edge incident with $v$ as follows.  If $vu$ is an edge of colour $i$ in $G$, then $vu$ is an edge of colour $\pi(i)$ in $G^{(v,\pi)}$.  We call this operation \emph{switching at $v$ with respect to $\pi$}.  Given a sequence $\Sigma = (v_1, \pi_1), (v_2, \pi_2), \dots, (v_k, \pi_k)$, \emph{switching with respect to $\Sigma$} is defined as

	$$
	G^{\Sigma} = (G^{(v_1,\pi_1)})^{(v_2,\pi_2), \dots, (v_k, \pi_k)}.
	$$

We call $\Sigma$ a \emph{switching sequence}. Whenever $G'$ can be obtained from $G$ through a sequence of switches, i.e., $G'=G^{\Sigma}$, we say $G$ and $G'$ are \emph{\sweq}, denoted $\swequiv{G}{G'}$.  (Note that the underlying graph of $G$ and $G'$ is the same.)  Of particular importance is the case when $\Gamma$ is Abelian.  Clearly if we switch consecutively at the same vertex, we may combine the switches into a single switch, i.e., $G^{(v,\pi_1)(v,\pi_2)} = G^{(v,\pi_2 \pi_1)}$.  In particular, if $\Gamma$ is Abelian, we can reorder any switching sequence $\Sigma$ so that for any vertex $v$, all switches on $v$ are consecutive in $\Sigma$ and thus can be combined into a single switch.  This is an important fact we use below and thus we record in here.

\begin{proposition}\label{prop:atmostonce}
Let $G$ be an $m$-edge-coloured graph and let $\Gamma \subseteq S_m$ be an Abelian group.  Suppose $\Sigma$ is a switching sequence.  Then there is a switching sequence $\Sigma'$ such that
\[
G^{\Sigma} = G^{\Sigma'}
\]
and under $\Sigma'$ each vertex of $G$ is switched at most once.
\end{proposition}

The polynomial-time results below require that $\Gamma$ acts transitively on $\{1, 2,$ $ \dots, m \}$; hence, for ease of terminology, we call a  group $\Gamma$ an \emph{$m$-switching group} if $\Gamma \subseteq S_m$  and $\Gamma$ acts transitively on $\{1, 2, \ldots, m\}$.

\subsection{Switching homomorphisms and CSPs}\label{subsec:CSP}

We have defined $(m,n)$-mixed graphs to be constructed through assigning colours and orientations to edges of an underlying graph.  Equivalently, an $(m,n)$-mixed graph may be viewed as a set of vertices $V$ together with $m$ symmetric binary relations and $n$ anti-symmetric binary relations on $V$. Extending this idea further, a \emph{relational system} $\mathbf{H}$ is a set of vertices, $V(H)$, together with $k$ relations where relation $R_i$ has arity $\sigma_i$.  The tuple $(\sigma_1, \dots, \sigma_k)$ is the \emph{signature} of $\mathbf{H}$. A homomorphism of relational systems $\mathbf{G}$ to $\mathbf{H}$ (over the same signature) is naturally defined as a vertex mapping which preserves all relations. For a fixed relational system $\mathbf{H}$, the \hhom[$H$] problem from the graph setting naturally generalizes to the homomorphism problem for $\mathbf{H}$, denoted \hhom[$\mathbf{H}$]. A \emph{Constraint Satisfaction Problem} (CSP) is a problem that can be stated as \hhom[$\mathbf{H}$] for some fixed relational system $\mathbf{H}$.

The settling of the CSP dichotomy conjecture~\cite{fed93} by Bulatov~\cite{bul17} and Zhuk~\cite{zhu20} gives that \hhom[$\mathbf{H}$]\ is either polynomial-time solvable or NP-complete for any fixed $\mathbf{H}$.  The description of the dichotomy is through symmetries and algebra. For many restricted families of targets, there is a  structural (combinatorial) description of the \hhom[$\mathbf{H}]$ dichotomy for that family. Such a dichotomy for the family of all undirected graphs is the Hell-Ne\v{s}et\v{r}il Dichotomy Theorem given above. On the other hand, in~\cite{fed98}, it is shown that for each CSP, there is a directed acyclic graph (DAG) $H$ such that the CSP and \hhom\ are polynomial-time equivalent.  Thus a structural  description of the dichotomy for \hhom\ when restricted to DAGs is likely complicated.  
In~\cite{bre17} the authors modify the proof of~\cite{fed98} to show for each CSP there is a $2$-edge-coloured graph $H$ with a polynomial-time equivalent \hhom\ problem.  Consequently, a structural description of the homomorphism dichotomy for $(0,1)$-mixed graphs or $(2,0)$-mixed graphs is likely complicated.  

In contrast to general CSPs, the introduction of a switching operation seems to lead to simpler homomorphism dichotomy descriptions.  For oriented graphs this has been observed in~\cite{klo04} and for $2$-edge-coloured graphs, i.e., \emph{signed graphs}, a simple dichotomy description is given in~\cite{bre17, bre18}. 

We now formally define switching homomorphisms. Given $m$-edge-coloured graphs $G$ and $H$, a \emph{homomorphism} of $G$ to $H$ is a function $f:V(G) \to V(H)$ such that if $uv$ is an edge of colour $i$ in $G$, then  $f(u)f(v)$ is an edge of colour $i$ in $H$. We sometimes use the notation $f: G \to H$ to denote a homomorphism $f$ of $G$ to $H$, or  $G \to H$ when the name of the function is not required.  We say $G$ admits a \emph{\swhomom} to $H$, and write $\swhom{G}{H}$, if $G' \to H$ for some $G'$ which is \sweq\ to $G$.
If $\swhom{G}{H}$ and $H'$ is \sweq\ to $H$, then it is easy to verify that 
$\swhom{G}{H'}$ (apply switches on $H$ to pre-images in $G$).  
A similar argument is used to prove the following proposition, which has previously appeared in several contexts.

\begin{proposition}[\cite{klo99, bre09, bre17, macwar21}]\label{prop:compose}
Let $H_1, H_2$ and $H_3$ be $m$-edge-coloured graphs, and let $\Gamma$ be an $m$-switching group.
    If $f$ is a \swhomom\ of $H_1$ to $H_2$ and $g$ is a \swhomom\ of $H_2$ to $H_3$, then 
    $g \circ f$ is a \swhomom\ of $H_1$ to $H_3$.
\end{proposition}

We conclude this section with a remark on the restriction to $m$-edge-coloured graphs and transitive groups.  The equivalence to DAGs in~\cite{fed98} shows that a structural dichotomy theorem even for $(0,1)$-mixed graphs with a trivial switching group $\Gamma$ is likely complicated.  Modifying the proof of~\cite{fed98} shows that the case where $\Gamma$ is non-transitive is also likely to be complicated as is the case where $G$ has both edges and arcs.  Hence, we believe our restrictions here are natural.

\subsection{Main Results and Organization}

The focus of this paper is the following decision problem.  Let $H$ be a fixed $m$-edge-coloured graph and $\Gamma$ be a fixed permutation group acting on $\{ 1, 2, \dots, m \}$.
	
	\begin{quote}
		\swhhom{$H$} \\
		\textbf{Instance:} An $m$-edge-coloured graph $G$. \\
		\textbf{Question:} Is there a \swhomom\ of $G$ to $H$?
	\end{quote}		

To state our main theorem we require the following standard notion.  Two $m$-edge-coloured graphs $H_1$ and $H_2$ are \emph{$\Gamma$-homomorphically equivalent} if there exist homomorphisms $\swhom{H_1}{H_2}$ and  $\swhom{H_2}{H_1}$. This defines an equivalence relation on the set of $m$-edge-coloured graphs. If $H_1$ and $H_2$ are $\Gamma$-homomorphically equivalent, then by Proposition~\ref{prop:compose} there is a \swhomom\ of an $m$-edge-coloured graph $G$ to $H_1$ if and only if there is a \swhomom\ of $G$ to $H_2$.
After noting that every bipartite (classical) graph is homomorphically equivalent to $K_1$ or $K_2$, the Hell-Ne\v{s}et\v{r}il Dichotomy Theorem can be restated as:
\emph{If the fixed graph $H$ is homomorphically equivalent to $K_1$ or $K_2$, then the \hhom[$H$] problem is polynomial-time solvable; otherwise, it is NP-complete.}  Our main result is the analogous statement for $m$-edge-coloured graphs holds in the case that $\Gamma$ is an $m$-switching group.  Let $K_2^{i}$ be the $m$-edge-coloured graph with underlying graph equal to $K_2$ and its unique edge having colour $i$.

\begin{theorem}\label{thm:main}
Let $H$ be a fixed $m$-edge-coloured graph and let $\Gamma$ be a fixed $m$-switching group.  If $H$ is $\Gamma$-homomorphically equivalent to $K_1$ or $K_2^{i}$ (for any $i$), then the \swhhom{$H$} problem is polynomial-time solvable; otherwise, it is NP-complete.
\end{theorem}

The paper is organized as follows.  In Section~\ref{sec:abdic} we restrict to the case that $\Gamma$ is Abelian. We introduce the switch graph which allows us to show that \swhhom{$H$} is indeed a CSP and thus the problem admits a dichotomy. We then establish Theorem~\ref{thm:main} for the Abelian case.

In Section~\ref{sec:groupab} we settle the case when $\Gamma$ is non-Abelian.  Here it is not clear that \swhhom{$H$} is a CSP nor that it is even in NP. Our contribution in this section is to show for each $m$-switching group $\Gamma$ there is an integer $m'$, an Abelian $m'$-switching group $\Gamma'$, and an edge-coloured graph $H'$ such that \swhhom{$H$}\  and \swhhom[$\Gamma'$]{$H'$} are polynomially equivalent.  By reducing to the Abelian case we establish Theorem~\ref{thm:main} for all fixed simple $m$-edge-coloured graphs $H$ and all $m$-switching groups $\Gamma$.  

\section{A dichotomy for Abelian $\Gamma$}\label{sec:abdic}

In this section we establish a dichotomy theorem for the \swhhom{$H$} problem when the $m$-switching group $\Gamma$ is Abelian. 

We first note it is not immediately clear whether testing the existence of a \swhomom\ to a fixed target is indeed a CSP as each CSP is equivalent to testing the existence of a homomorphism to a fixed target $\mathbf{H}$. Where as, a \swhomom\ is a two step process: switch $G$ to some $G'$, then map $G' \to H$.  The following construction has appeared in many contexts~\cite{klo99, bre09, bre17, kid21, macwar21}.  Given an $m$-edge-coloured graph $H$ and a fixed \emph{Abelian} permutation group $\Gamma$ acting (perhaps not transitively) on $\{ 1, 2, \dots, m \}$, the \emph{$\Gamma$-switch graph}, $\swg{H}$, has vertices $V(H) \times \Gamma$ and edges defined as follows.  If $v_j v_k$ is an edge of colour $i$ in $H$, then $(v_j, \pi_r)(v_k, \pi_s)$ is an edge of colour $\pi_r\pi_s(i)$ in $\swg{H}$.  As a convention we shall use $e_\Gamma$ to denote the identity of $\Gamma$.  In particular, the vertices of the form $(v, e_\Gamma)$ induce a copy of $H$ in $\swg{H}$.
The following proposition shows that the existence of a $\Gamma$-homomorphism $\swhom{G}{H}$ is equivalent to the existence of a homomorphism $G \to \swg{H}$ (without switching).
	
	\begin{proposition}[\cite{klo99, bre09, bre17, kid21, macwar21}] \label{prop:withoutSwitch}
	Let $G$ and $H$ be $m$-coloured graphs and $\Gamma \subseteq S_m$ be an Abelian group acting on $\{ 1, 2, \dots, m \}$. The following are equivalent.
	
	\begin{list}{(\alph{enumi})}{\usecounter{enumi}}
	  \setlength{\itemsep}{0pt}
	  \item $\swhom{G}{H}$,
	  \item $G \to \swg{H}$,
	  \item $\swg{G} \to \swg{H}$.
	\end{list}
	\end{proposition}

As cycles play an important role in this dichotomy we first determine important facts about switching cycles with Abelian groups. 

We say a graph or subgraph (usually a cycle or spanning tree) is \emph{monochromatic of colour $i$} if all of its edges are of colour $i$.  In the case that a cycle has all edges but one of colour $i$ and one edge of colour $j \neq i$, we say that it is \emph{nearly monochromatic of colours $(i,j)$}. When no confusion arises, we may omit the colours and say the graph or subgraph is monochromatic or nearly-monochromatic. These concepts generalize the situation for signed graphs to $m$-edge-coloured graphs 
(see~\cite{Naserasr_SZ_2021} and the references  within). 

A key step in our polynomial results is testing whether an $m$-edge-coloured graph $G$ can be switched to be monochromatic. We prove the following useful result that the answer is always yes when $G$ is a tree.

\begin{proposition}\label{prop:switchingtree}
    Let $T$ be an $m$-edge-coloured tree and let $\Gamma$ be an $m$-switching group.  Then for any colour $i \in \{1, \dots, m \}$, $T$ maybe switched to be monochromatic of colour $i$.
\end{proposition}

\begin{proof}
    Root $T$ at some vertex $r$. Using a BFS ordering of $V(T)$, say $r=v_1, \dots, v_n$, we switch $v_t$ so that the edge from $v_t$ to its parent is of colour $i$ for $t=2, \dots, n$.
\end{proof}

Note this proposition also applies to spanning trees, i.e., if $T$ is a spanning tree of $G$, then we may switch $G$ so that $T$ is monochromatic of colour $i$.

To prove the following proposition we first require a well-known fact about Abelian group actions (the proof of which we include here for completeness). Suppose $\Gamma$ is an Abelian $m$-switching group.  Let $\pi \in \Gamma$ fix $i \in \{ 1, \dots, m \}$, i.e., $\pi(i) = i$.  For any $j \in \{ 1, \dots, m \}$, by transitivity, there is $\tau \in \Gamma$ such that $\tau(i) = j$.  Then,
\[
\pi(j) = \pi(\tau(i)) = \tau(\pi(i)) = \tau(i) = j.
\]
Thus, if $\pi$ fixes one element of $\{ 1, \dots, m \}$, then $\pi$ fixes all elements of $\{ 1, \dots, m \}$.  The only element of $S_m$ that fixes all of $\{ 1, \dots, m \}$ is the identity $e_{S_{m}}$  (the action of $\Gamma$ is \emph{faithful}).  We record the following fact.

\begin{fact}\label{fact:GammaFree}
If $\Gamma$ is an $m$-switching group such that $\pi(i)=i$ for some $i \in \{1, \dots, m \}$, then $\pi = e_{\Gamma}$.
\end{fact}

Further, using Fact~\ref{fact:GammaFree}, the action of $\Gamma$ has the following property. If for two edge colours $i$ and $j$ we have $\pi(i) = j$ and $\sigma(i)=j$, then $\sigma^{-1}\pi(i) = i$ which implies $\sigma^{-1}\pi = e_{\Gamma}$ or $\pi = \sigma$.  That is, there is exactly one element of $\Gamma$ taking $i$ to $j$.  Such actions are called \emph{regular}.

Armed with Fact~\ref{fact:GammaFree}, we now prove the following result on switching cycles.

\begin{proposition}\label{prop:monocycles}
Suppose $\Gamma$ is an Abelian $m$-switching group and let $C$ be an $m$-edge-coloured cycle.
Then exactly one of the following statements holds:
\begin{list}{(\roman{enumi})}{\usecounter{enumi}}
	\setlength{\itemsep}{0pt}
	\item $C$ has even length and for any colour $i$, $C$ can be switched using $\Gamma$ to be monochromatic of colour $i$;
 	\item $C$ has even length and for any colour $i$, $C$ can be switched using $\Gamma$ to be nearly monochromatic of colours $(i,j)$ for some colour $j$; or
    \item $C$ has odd length and for some colour $j$, $C$ can be switched using $\Gamma$ to be monochromatic of colour $j$.
	\end{list}
\end{proposition}

\begin{proof}
Let the cycle $C$ have vertices $v_0, v_1, \dots, v_{n-1}$ in the natural order. By Proposition~\ref{prop:switchingtree}, for any colour $i$, we can switch the path $v_1, v_2, \dots, v_{n-1}$ to be monochromatic of colour $i$.  If after this switching, $v_0v_{n-1}$ is also of colour $i$, then $C$ has been switched to be monochromatic.  Otherwise, $v_0v_{n-1}$ is of colour $j$ in which case $C$ is nearly monochromatic of colour $(i,j)$. Thus, (i) or (ii) holds. 

We now show that statements (i) and (ii) are mutually exclusive.  Suppose (after the switching above) $C$ is nearly monochromatic of colours $(i,j)$.
Further suppose to the contrary that it is possible to switch $C$ to be monochromatic. We first observe that if $C$ can be switched to be monochromatic of colour $k$, then by transitivity there is a permutation $\pi(k)=i$ which when applied to $v_0, v_2, v_4, \dots v_{n-2}$ switches to $C$ to be monochromatic of colour $i$.  So it suffices to prove $C$ cannot be switched to be monochromatic of colour $i$.

By Proposition~\ref{prop:atmostonce}, since $\Gamma$ is Abelian, we can assume that at most one switch is applied at each vertex. Let $(\pi_s, v_s)$ be the switch applied to $v_s$ for $0 \leq s \leq n-1$ (including the possibility that $\pi_s = e_{\Gamma}$ when there is no switch). For each edge $v_sv_{s+1}$, $0 \leq s \leq n-2$, we have $\pi_{s+1}\pi_s(i)=i$, and thus $\pi_{s+1}=\pi^{-1}_s$ by Fact~\ref{fact:GammaFree}.  This implies $\pi_0$ is applied at all vertices with an even subscript and $\pi^{-1}_0$ at all vertices with an odd subscript. However this implies the edge $v_0v_{n-1}$ is switched from colour $j$ to colour $i$ by $\pi_0\pi^{-1}_0$, a contradiction. It follows that if $C$ has even length then exactly one of (i) or (ii) holds.

Finally, suppose $C$ is an odd cycle and has been switched so that $v_0v_1, v_1v_2,$ $\dots, v_{n-2}v_{n-1}$ are colour $i$. If $C$ is monochromatic, then we are done. Otherwise, $v_0v_{n-1}$ is colour $j \neq i$. By transitivity there is $\pi \in \Gamma$ satisfying $\pi(i)=j$.  Switching at vertices $v_1, v_3, \dots, v_{n-2}$ with respect to $\pi$ switches $C$ to be monochromatic of colour $j$.

Since statements (i), (ii) and (iii) are mutually exclusive, the result follows.
\end{proof}

We now present the dichotomy result for the Abelian case. Denote by $K_2^i$ the monochromatic $m$-edge-coloured graph of colour $i$ whose underlying graph is $K_2$. The polynomial cases of the dichotomy theorem rely on the following theorem. It was the first non-trivial result which applies to all groups that act transitively on the edge colours.

\begin{theorem}[\cite{bre_kid_mac21, kid21}]
Let $\Gamma$ be an $m$-switching group and $i \in \{1, 2, \ldots, m\}$. 
Then \hhom[$K_2^i$] is polynomial-time solvable.
\label{thm:Arnott}
\end{theorem}

We provide a sketch of the proof here for the case that $\Gamma$ is Abelian. If the given input $m$-edge-coloured graph $G$ is not bipartite, then we have a no-instance.  Otherwise, assume $G$ is a connected, bipartite $m$-edge-coloured graph. Fix a spanning tree $T$ of $G$ and switch $T$ to be monochromatic of colour $i$ using Proposition~\ref{prop:switchingtree}. 

Next we look at the co-tree edges. If each co-tree edge is of colour $i$, then $G$ has be switched to be monochromatic of colour $i$. As a (classical) bipartite graph admits a homomorphism to $K_2$,  $\swhom{G}{K_2^i}$, i.e., $G$ is a yes-instance. Otherwise, some co-tree edge is of colour $j\neq i$ and so $G$ contains a cycle that is $\Gamma$-equivalent to a nearly monochromatic cycle of colour $(i,j)$. By Proposition~\ref{prop:monocycles}(ii), the cycle cannot be switched to be monochromatic of colour $i$ and hence neither can $G$.
Thus $G$ is a no-instance.

Alternatively, we can leverage the switch graph using Proposition~\ref{prop:withoutSwitch}. In the remarks following Fact~\ref{fact:GammaFree}, we noted that the action of $\Gamma$ is regular.  This implies for each vertex of $\swg{K_2^i}$, it is incident with exactly one edge of each colour.
Thus one can test $\swhom{G}{K_2^i}$ by first mapping one vertex $v$ of $G$ to some vertex $(w,e_\Gamma)$ in $\swg{K_2^i}$. 
Each neighbour of $v$ in $G$ must then map to the unique neighbour of $(w,e_\Gamma)$ joined by an edge of the same colour. Continuing, we either construct a homomorphism or obtain a contradiction. This approach avoids the need to switch. 
If the answer is ``no'' then one can trace backwards to find a certificate that no homomorphism exists. (The certificate will in fact be a cycle in $G$ that cannot be switched to be monochromatic c.f. Proposition~\ref{prop:monocycles}.)

The NP-complete cases of the dichotomy theorem use the following results; namely, the indicator construction~\cite{hel08} and a theorem of Barto, Kozig and Niven~\cite{bar08}.

Let $A$ be an $m$-edge-coloured graph with specified vertices $a$ and $b$.
Given an $m$-edge-coloured graph $H$, the \emph{indicator construction with respect to $(A, a, b)$} produces a digraph $H^*$ with vertex set $V(H^*) = V(H)$ and an arc from $x$ to $y$ whenever there is a homomorphism (without switching) of $A$ to $H$ that maps $a$ to $x$ and $b$ to $y$.

\begin{lemma}[\cite{hel08}]
    \hhom[$H^*$] polynomially transforms to \hhom.
    \label{indicator}
\end{lemma}

A digraph $D$ is called \emph{smooth} if it has no vertex of in-degree 0 and no vertex of out-degree 0.  A dichotomy theorem for homomorphisms to smooth digraphs is known:

\begin{theorem}[\cite{bar08}]
    Let $H$ be a smooth digraph.  If each component of $H$ is homomorphically equivalent to a directed cycle, then \hhom\ is polynomial-time sovable; otherwise, \hhom\ is NP-complete.
    \label{smooth}
\end{theorem}

\begin{corollary}
    Let $H$ be a smooth digraph with no loops.  If some component of $H$ contains directed cycles whose lengths are relatively prime, then \hhom\ is NP-complete.
    \label{smoothcor}
\end{corollary}

 We now state and prove the main result of this section.

\begin{theorem}\label{thm:abelianP}
Let $\Gamma$ be an Abelian $m$-switching group.  For a fixed $m$-edge-coloured graph $H$, if
$H$ is $\Gamma$-homomorphically equivalent to a monochromatic bipartite $m$-edge-coloured graph,
then \swhhom{$H$} is polynomial-time solvable. 
Otherwise, \swhhom{$H$} is NP-complete.
\end{theorem}

\begin{proof}

The $m$-edge-coloured graph $H$ is $\Gamma$-homomorphically equivalent to a monochromatic bipartite graph if and only if $H$ has a \swhomom\ to a monochromatic $K_1$ or a monochromatic $K_2$, the former case occurring when $H$ has no edges.  Thus, testing if some input $m$-edge-coloured graph $G$ admits a \swhomom\ to $H$ is reduced to the cases when $H$ is a $K_1$ or $K_2^i$ for some $i$.  
The case for $K_1$  is easy to test, simply verify $G$ has no edges.  
In the case that $H = K_2^i$, \swhhom{$H$} is polynomial-time solvable by Theorem \ref{thm:Arnott}. 

Now suppose $H$ is $\Gamma$-homomorphically equivalent to neither a 
monochromatic $K_1$ nor a monochromatic $K_2$.
Thus either $H$  is bipartite and has some component that cannot be switched to be monochromatic, or $H$ (has a component that) is not bipartite.  We note as $H$ is homomorphically equivalent to the subgraph induced by non-trivial components (and in this case we must have non-trivial components) we may assume without loss of generality that $H$ has only non-trivial components.

Suppose $H$ is bipartite and some component $H_1$ cannot be switched to be monochromatic. Choose a spanning tree $T$ of $H_1$.  By Proposition~\ref{prop:switchingtree} there is a sequence of switches $\Sigma$ so that in $H_1^\Sigma$ every edge of $T$ has colour $i$. Since $H_1$ cannot be switched to be monochromatic, there must be some co-tree edge of colour $j \neq i$. Let $C$ be a nearly monochromatic cycle in $H_1^\Sigma$ on vertices $v_0, v_1, \dots, v_{2k-1}$ with all edges in $T$ except the co-tree edge $v_0 v_{2k-1}$, which has colour $j \neq i$.  

For ease of notation denote $\swg{H^\Sigma}$ by $SH$.  
Let $A$ be the $2$-edge coloured graph consisting of a path of length $2$ with end points $a$ and $b$ with the edge incident with $a$ of colour $i$ and the edge incident with $b$ of colour $j$. Let $SH^*$ be the result of applying the indicator construction with respect to $(A,a,b)$ to $SH$. By Lemma~\ref{indicator}, we need to show the \hhom[$SH^*$]\ problem is NP-complete. This is accomplished by proving $SH^*$ is smooth, and some component has directed cycles of relatively prime lengths.

\begin{claim}
	The digraph $SH^*$ is smooth.
\end{claim}

\begin{claimpf}\claimproof
As each component of $SH$ is non-trivial, each vertex $u$ in $H$ is incident with at least one edge.
Then each vertex $(u,\pi)$ in $SH$ is incident with an edge $(u,\pi)(w,\pi)$ of some colour, say $c$.  Since $\Gamma$ acts transitively, there is $\tau \in \Gamma$ such that $\tau(c) = i$, or equivalently the edge $(u,\pi)(w,\tau\pi)$ is of colour $i$.  Similarly the vertex $(w,\tau\pi)$ is incident with an edge of colour $j$, say $(w,\tau\pi)(v,\sigma)$. Hence there exists a homomorphism from $A$ to $SH$ where $a$ is mapped to $(u,\pi)$ and $b$ is mapped to $(v, \sigma)$. We conclude that in $SH^*$, the vertex $(u,\pi)$ is the tail of an arc.  Using similar reasoning, the vertex $(u,\pi)$ is also the head of an arc in $SH^*$.  This proves the claim.
\end{claimpf}

\begin{claim}
	The digraph $SH^*$ has a component containing directed cycles of relatively prime lengths.
\end{claim}

\begin{claimpf}\claimproof
Recall, $SH$ contains a copy of $H^\Sigma$ and thus by hypothesis, there is a component of $SH$ containing the nearly monochromatic cycle $C$ in $H_1^{\Sigma}$.
\[
(v_0, e_\Gamma), (v_1, e_\Gamma), \dots, (v_{2k-1}, e_\Gamma), (v_0, e_\Gamma)\]
with all edges of colour $i$ except the edge $(v_0, e_\Gamma)(v_{2k-1}, e_\Gamma)$ which is of colour $j$.  
	
Since $\Gamma$ acts transitively, there exists $\pi \in \Gamma$ such that $\pi(i) = j$. Let the order of $\pi$ be $d$, that is, $\pi^d = e_{\Gamma}$. Observe that the edge $(v_0, \pi^s)(v_1, \pi^t)$ will be of colour $i$ when $s+t \equiv 0 \pmod{d}$ and will be of colour $j$ when $s+t \equiv 1 \pmod{d}$.
	
In $SH$ the cycle
\begin{eqnarray*}
	C_1 & = & (v_0,e_\Gamma), (v_1,e_\Gamma), (v_0,\pi), (v_1,\pi^{d-1}), (v_0,\pi^2), (v_1,\pi^{d-2}), \\
	& &  \ldots, (v_0,\pi^\ell), (v_1,\pi^{d-\ell}), \ldots, ( v_0,\pi^{d-1}), (v_1,\pi), (v_0,e_\Gamma)
\end{eqnarray*} 
is a cycle of length $2d$ whose edges alternate colour $i$ and colour $j$.  
An example is shown in Figure~\ref{fig:littlecycle}.
\begin{figure}[ht]%
	\centering
	
	\includegraphics[width=0.95\textwidth]{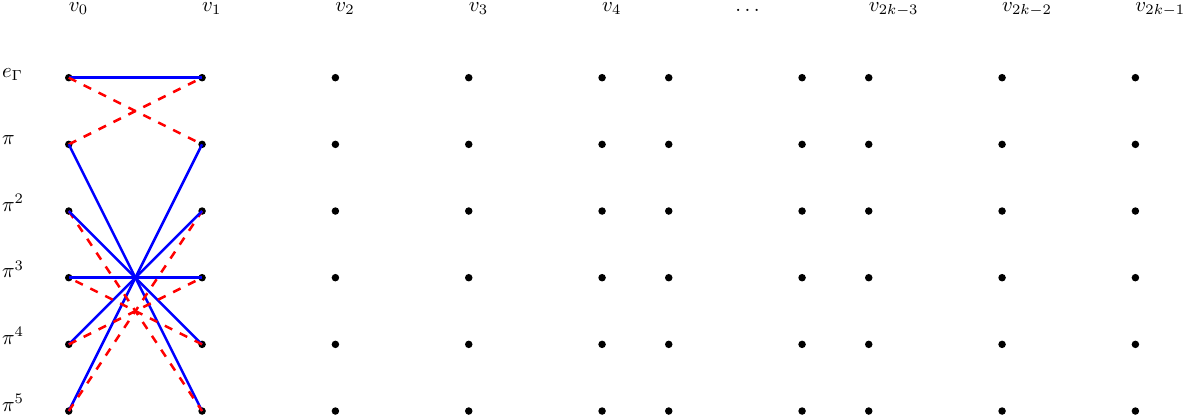}
	
	\caption{An $(i,j)$-alternating cycle of length $12$ in $SH$ where $\pi$ has order $6$.}\label{fig:littlecycle}
\end{figure}

For $r$ odd and $1\leq r \leq 2k-3$, let $P_r$ be the path given by
\begin{eqnarray*}
    P_r &=& (v_r,e_\Gamma), (v_{r+1},\pi), (v_r,\pi^{d-1}), \ldots, (v_{r+1},\pi^\ell), (v_r,\pi^{d-\ell}),\\ 
	& &  \ldots, ( v_{r+1},\pi^{d-1}), (v_r,\pi), (v_{r+1}, e_\Gamma), (v_{r+2}, e_\Gamma).
\end{eqnarray*}

This path has length $2d$ and edges that alternate colour $i$ and colour $j$. And so, the cycle
\begin{eqnarray*}
    C_2 &=& (v_0,e_\Gamma), P_1, P_3, \ldots, P_{2k-3}, (v_{2k-1},e_\Gamma), (v_0,e_\Gamma)
\end{eqnarray*}
consists of $k-1$ paths of length $2d$ and the path $(v_{2k-1},e_\Gamma),(v_0,e_\Gamma)$ of length $2$. Hence, $C_2$ is a cycle of length $(2d)(k-1) + 2$ whose edges alternate colour $i$ and colour $j$.  An example is shown in Figure~\ref{fig:bigcycle}.

\begin{figure}[ht]%
	\centering
	
	\includegraphics[width=0.95\textwidth]{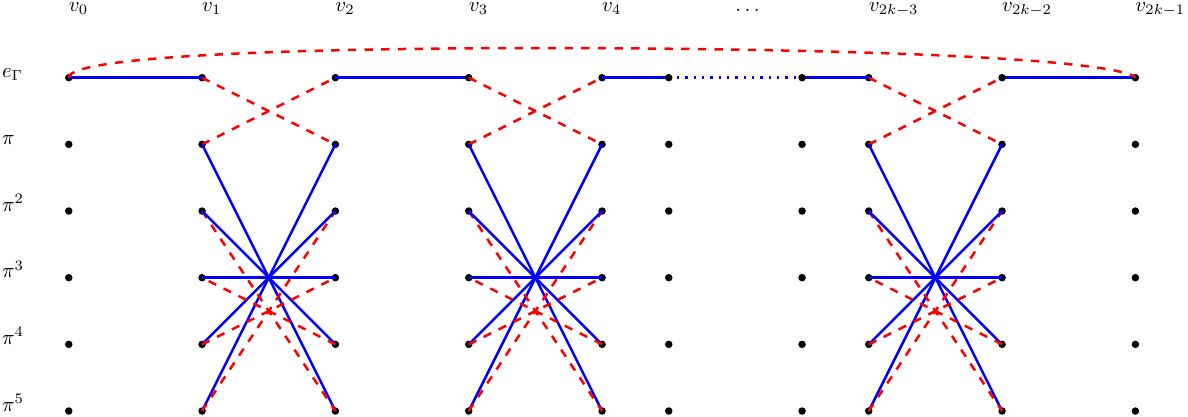}
	
	\caption{An $(i,j)$-alternating cycle of length $12(k-1)+2$ in $\swg{H}$ where $\pi$ has order $6$.}\label{fig:bigcycle}
\end{figure}

The cycles $C_1$ and $C_2$ show that $SH^*$ has a component containing directed cycles of relatively prime lengths $d$ and $d(k-1)+1$ (on the vertices of $C_1$ and $C_2$ at even 
distance along the cycle from $(v_0,e_\Gamma)$), which completes the proof of the claim.
\end{claimpf}
	
As $SH$ does not contain any cycles of length $1$ or $2$, we have $SH^*$ is loop-free.  Thus \hhom[$SH^*$] is NP-complete by Corollary \ref{smoothcor}, and
hence \hhom[$SH$] is NP-complete.

Finally, suppose $H$ is not bipartite.  Thus it contains an odd cycle $C$.
By Proposition~\ref{prop:monocycles}(iii), $H$ is $\Gamma$-switch equivalent to a graph with a monochromatic odd cycle. Thus, $\swg{H}$ contains a monochromatic odd cycle, say of colour $i$.  By~\cite{hel90}, the \hhom[$\swg{H}$]\ problem is NP-complete by restricting the inputs to colour $i$.  By Proposition~\ref{prop:withoutSwitch}, \swhhom{$H$}\ is NP-complete.
  
This completes the proof of the theorem.
\end{proof}
	
\section{Group Abelianization}\label{sec:groupab}

We now turn our attention to non-Abelian $m$-switching groups. Some special cases have been considered in earlier work \cite{mactre21}. 

Let $H$ be an $m$-edge-coloured graph and let $\Gamma$ be a  $m$-switching group. In this section we show there is a set of $m'$ colours,  a $m'$-edge coloured graph $H'$ with the same underlying graph as $H$, and an Abelian $m'$-switching group $\Gamma'$ such that \swhhom[$\Gamma'$]{$H'$} and \swhhom{$H$}\ are polynomially equivalent.

We recall some group theory terminology. For $\pi,\phi\in \Gamma$, the element $[\pi, \phi] = \pi\phi\pi^{-1}\phi^{-1}$ is a \emph{commutator} of $\Gamma$. The \emph{commutator subgroup of $\Gamma$}, denoted $[\Gamma, \Gamma]$, is the subgroup generated by the commutators of $\Gamma$.  We now show how commutators can be used to generate a sequence of switches that allow the colour of a single edge to be changed.
	
\begin{lemma}\label{lem:switchone}
Let $H$ be an $m$-edge-coloured graph and $\Gamma$ be an $m$-switching group.  For any $\tau \in [\Gamma, \Gamma]$ and any edge $uw$ in $H$ of colour $i$, there is a sequence $\Sigma$ of switches such that in $H^\Sigma$ the edge $uw$ has colour $\tau(i)$ and all other edges in $H$ and $H^\Sigma$ have the same colour.
\end{lemma}

\begin{proof}
Since $\tau \in [\Gamma, \Gamma]$, we have $\tau = [\pi_1, \phi_1] \cdots [\pi_k, \phi_k]$.  Let
$$
H^\Sigma = H^{(w, \phi_k^{-1})(u, \pi_k^{-1})(w,\phi_k)(u,\pi_k) \cdots (w,\phi_1^{-1})(u,\pi_1^{-1})(w, \phi_1)(u, \pi_1)}.
$$
The switches at $u$ are the product $\pi_k^{-1} \pi_k \cdots \pi_1^{-1} \pi_1$ and at $w$ are $\phi_k^{-1} \phi_k \cdots \phi_1^{-1} \phi_1$.  Thus each edge other that $uw$ has the same colour in $H$ and $H^\Sigma$.  The edge $uw$, on the other hand, has colour $\tau(i)$ as required. 
\end{proof}

It is well known that $[\Gamma, \Gamma]$ is the smallest normal subgroup of $\Gamma$ such that the quotient group $\Gamma/{[\Gamma,\Gamma]}$
is Abelian (see \cite{Gallian_2017}, Chapter 9, exercise 62).  
The orbits of $[\Gamma, \Gamma]$  partition  $\{ 1, 2, \dots, m \}$ as $D = \{\Delta(1), \ldots, \Delta(m)\}$, where $j \in \Delta(i)$ if there is $\tau \in [\Gamma, \Gamma]$ such that $\tau(i) = j$.  Note that $\Delta(i) = \Delta(j)$ if and only if $j \in \Delta(i)$. 
The partition $D$ is a block system for $\Gamma$ (see \cite{dix96}, Theorem 1.6A). 

Define an action of $\Gamma/{[\Gamma,\Gamma]}$ on $D$ by $g[\Gamma,\Gamma](\Delta) = g(\Delta)$, for $\Delta \in D$.  
This action is well defined because $h(\Delta) = \Delta$ for every $h \in [\Gamma,\Gamma]$ and $\Delta \in D$.
Since $\Gamma$ acts transitively on $\{ 1, 2, \dots, m \}$, it follows that the quotient group acts transitively on $D$.

Let $m'$ be the number of distinct blocks in $D$. 
In the sequel we will regard these as edge colours.
 
Given an $m$-edge-coloured graph $G$, we define the \emph{Abelianization of $G$ with respect to $\Gamma$}, denoted $G^{\Gamma}$, to be the $m'$-edge coloured graph obtained from $G$ by replacing each edge of colour $i$ with an edge of colour $\Delta(i)$.  

The quotient group $\Gamma/{[\Gamma,\Gamma]}$ acts transitively on the set of $m'$ colours.
It may not be a $m'$-switching group because it may have several permutations that act identically on the block system. 
To see this, consider $\Gamma = S_m$.
Then $[\Gamma,\Gamma] = A_m$ and $\Gamma/{[\Gamma,\Gamma]} = S_2$.  But in this case $m' = 1$ (since 
$A_m$ acts transitively), hence the quotient group is 
not a $m'$-switching group as it is not a subgroup of $S_{m'}$.
However, both elements of the quotient group act identically on the block system.

Let $\Gamma^{ab}$ be the subgroup of $S_{m'}$ which is isomorphic to the subgroup of $\Gamma/{[\Gamma,\Gamma]}$
arising from defining two elements of $\Gamma/{[\Gamma,\Gamma]}$ to be equivalent if they act
identically on the block system $D$.
Then $\Gamma^{ab}$ is a $m'$-switching group.
Since $\Gamma/{[\Gamma,\Gamma]}$ is Abelian, so is $\Gamma^{ab}$.
We call $\Gamma^{ab}$ the \emph{Abelianization of} $\Gamma$.

We now define switching on $G^\Gamma$ with respect to $\Gamma^{ab}$.
Recall that each $\pi' \in \Gamma^{ab}$ corresponds to an equivalence class of cosets of the form $\pi [\Gamma, \Gamma]$.  To switch $G^\Gamma$ at $u$ with respect to $\pi'$ we apply $\pi \in \Gamma$ to each of the edges incident with $u$ as follows: an edge $uv$ of colour $\Delta(i)$ in $G^{\Gamma}$ switches to colour $\Delta(\pi(i))$.  We claim that such switching is well defined in that for any element in the coset $\pi[\Gamma, \Gamma]$ and any colour $j \in \Delta(i)$, the resulting colour of $uv$ in $G^\Gamma$ is the same.  To see this, given $j \in \Delta(i)$ there is $\tau \in [\Gamma, \Gamma]$ such that $\tau(i) = j$.  Then $\pi(j) = \pi(\tau(i))$.  Since $[\Gamma, \Gamma]$ is a normal subgroup, $\pi\tau = \tau'\pi$ for some $\tau' \in [\Gamma, \Gamma]$.  Thus $\pi(j) = \tau'(\pi(i))$ giving $\pi(j) \in \Delta(\pi(i))$ as required.  Similarly, we can argue $\phi(i) \in \Delta(\pi(i))$ for any $\phi \in \pi[\Gamma, \Gamma]$ (or any other coset of $[\Gamma,\Gamma]$ with the same action on $D$).  This proves the claim.
As a result, if $\pi' = \pi[\Gamma, \Gamma]$, we can write $\pi'(\Delta(i)) = \Delta(\pi(i))$.

We note two special cases of interest.  If $\Gamma$ is Abelian, then $[ \Gamma, \Gamma] = \{ e_{\Gamma} \}$ and each block is a singleton, i.e., $\Delta(i) = \{ i \}$.  Hence, $\Gamma = \Gamma^{ab}$ and the action of the two groups is the same.  Conversely, if $[ \Gamma, \Gamma ]$ acts transitively on $\{ 1, 2, \dots, m \}$, then the block system consists of a single block, i.e. $\Delta(i) = \{ 1, 2, \dots, m \}$ for all $i$.  In this case for a given $G$, the $m$-edge-coloured graph $G^\Gamma$ has a single edge colour (and is invariant under switching by $\Gamma^{ab}$).  By Lemma~\ref{lem:switchone} it is possible to switch each edge of $G$ individually to any colour of $\{ 1, 2, \dots, m \}$ implying all edge colourings of $G$ are \sweq.  See \cite{mactre21} for examples.
	
\begin{theorem}\label{thm:switchAbel}
Let $G$ and $H$ be $m$-edge-coloured graphs and $\Gamma$ be an $m$-switching group. Then $\swequiv{G}{H}$ if and only if $\swequiv[\Gamma^{ab}]{G^{\Gamma}}{H^{\Gamma}}$.
\end{theorem}
	
\begin{proof}
Suppose $H = G^\Sigma$. By induction it suffices to consider the case $\Sigma = \{ (u, \pi) \}$.  Consider the edges incident with $u$. Suppose one such edge $uv$ has colour $i$ in $G$ and colour $j = \pi(i)$ in $H$.  In $G^{\Gamma}$ the edge $uv$ has colour $\Delta(i)$ and it has the colour $\Delta(j)$ in $H^{\Gamma}$. Letting $\pi' = \pi[\Gamma, \Gamma]$, we have $\pi'(\Delta(i)) = \Delta(\pi(i)) = \Delta(j)$.  This holds for all edges incident with $u$ and thus $(G^{\Gamma})^{(u,\pi')} = H^{\Gamma}$.

Now suppose $\swequiv[\Gamma^{ab}]{G^{\Gamma}}{H^{\Gamma}}$. As above it suffices to consider the case $H^{\Gamma} = (G^{\Gamma})^{(u,\pi')}$ for some $\pi' \in \Gamma^{ab}$. Again we consider the edges incident with $u$ and focus on an arbitrary edge $uv$.
Suppose in $G^{\Gamma}$ the edge $uv$ is colour $\Delta(s)$ and colour $\Delta(t)$ in $H^{\Gamma}$ where $\pi'(\Delta(s)) = \Delta(t)$.  By construction, in $G$ the edge $uv$ is some colour $i \in \Delta(s)$ and in $H$ some colour $j \in \Delta(t)$.

Suppose $\pi'$ is the coset $\pi[\Gamma, \Gamma]$. Since the switching operation is well defined, we have $\pi(i) = j'$ for some $j' \in \Delta(t)$. Thus, in $G^{(u,\pi)}$, the edge $uv$ has colour $j'$.  Since $j, j' \in \Delta(t)$, there is $\tau \in [\Gamma, \Gamma]$ such that $\tau(j') = j$. Using Lemma~\ref{lem:switchone}, there is a sequence $\Sigma$ that changes $uv$ from colour $j'$ to $j$ and leaves all other edges unchanged.  That is, in $G^{(u,\pi)\Sigma}$ the edge $uv$ has colour $j$.  After doing this for each edge incident at $u$, we see $\swequiv{G}{H}$.
\end{proof}
	
\begin{corollary}\label{cor:switchhomAbel}
Let $G$ and $H$ be $m$-edge-coloured graphs and $\Gamma$ be an $m$-switching group.  Then $\swhom{G}{H}$ if and only if $\swhom[\Gamma^{ab}]{G^{\Gamma}}{H^{\Gamma}}$.
\end{corollary}
	
\begin{proof}
Let $G$ and $H$ be $m$-edge-coloured graphs with Abelianizations $G^\Gamma$ and $H^\Gamma$ respectively.
Suppose $f: \swhom{G}{H}$. By definition, $f: G^{\Sigma} \to H$ for some sequence of switches $\Sigma$.  Consider an edge $uv$ of colour $i$ in $G$, and colour $i'$ in $G^{\Sigma}$.  Then by Theorem~\ref{thm:switchAbel} there is a sequence $\Sigma'$ of switches on $G^{\Gamma}$ so that edge $uv$ switches from $\Delta(i)$ to $\Delta(i')$.  In $H$, the edge $f(u)f(v)$ has colour $i'$ and thus in $H^{\Gamma}$ it has colour $\Delta(i')$.  Hence, the mapping $f: (G^{\Gamma})^{\Sigma'} \to H^{\Gamma}$ is a homomorphism as required.

Similarly, suppose $f: \swhom[\Gamma^{ab}]{G^{\Gamma}}{H^{\Gamma}}$, so that $f: (G^{\Gamma})^{\Sigma} \to H^{\Gamma}$.
Let $uv$ have colour $\Delta(s)$ in $G^\Gamma$ and colour $\Delta(t)$ in $(G^{\Gamma})^\Sigma$.  In $G$ the edge $uv$ has colour $i \in \Delta(s)$.  After applying the corresponding switches from $\Sigma$ (by selecting a representative from each coset) to $G$ to obtain $G^{\Sigma'}$, the edge $uw$ has colour $j$ for some $j \in \Delta(t)$.  In $H^\Gamma$, the edge $f(u)f(v)$ must have colour $\Delta(t)$ and hence $f(u)f(v)$ in $H$ has some colour $j' \in \Delta(t)$.  Thus, there is $\tau \in [\Gamma, \Gamma]$ such that $\tau(j) = j'$.  Hence, by Lemma~\ref{lem:switchone}, we can switch $G^{\Sigma'}$ so that $uv$ has colour $j'$.  Processing each edge in this way shows there is a sequence of switches $\Sigma''$ such that $\swequiv{G}{G^{\Sigma''}}$ and $f: G^{\Sigma''} \to H$ is a homomorphism as required.
\end{proof}

It is not immediately clear that the number of switches required to switch an $m$-edge-coloured graph $G$ to a $\Gamma$-equivalent $m$-edge-coloured graph $G^{\Sigma}$ is polynomial in the size of $G$.  If $\Gamma$ is Abelian, then there is a sequence that switches each vertex at most once.  As $\Gamma^{ab}$ is Abelian, we have the following corollary.

\begin{corollary}\label{cor:inNP}
Let $H$ be a fixed $m$-edge-coloured graph and let $\Gamma$ be a fixed $m$-switching group. Then \swhhom{$H$}\ is in NP.
\end{corollary}
	
\begin{proof}
As $\Gamma$ is fixed, the groups $\Gamma^{ab}$ and $[\Gamma, \Gamma]$ are of fixed sized.  We can certify $G$ is a yes-instance of \swhhom{$H$}\ by providing the switching sequence $\Sigma$ and the function $f: V(G) \to V(H)$.  By Corollary~\ref{cor:switchhomAbel}, we know the switches of $G$ consist of switches from $\Gamma^{ab}$ and switches from $[\Gamma, \Gamma]$.  As $\Gamma^{ab}$ is Abelian, each vertex needs to be switched at most once by an element of $\Gamma^{ab}$, and as $[\Gamma, \Gamma]$ is fixed, the sequence used to change the colour of a single edge (as in Lemma~\ref{lem:switchone}) is of bounded sized.  Thus, the switching sequence $\Sigma$ has length $O(|E(G)|)$.
\end{proof}	

We now state the full dichotomy theorem for \swhhom{$H$}.  
A fundamental step in the polynomial-time case is testing if a given $G$ can be switched to be monochromatic of some colour $i$.  
As $\Gamma^{ab}$ is Abelian, this can be done as described previously. 

\begin{theorem}\label{thm:fulldichotomy}
Let $\Gamma$ be an $m$-switching group.  For a fixed $m$-edge-coloured graph $H$, if
$H$ is $\Gamma$-homomorphically equivalent to a monochromatic bipartite graph,
then \swhhom{$H$} is polynomial-time solvable;  
otherwise, \swhhom{$H$} is NP-complete.
\end{theorem}

\begin{proof}
Note that $H$ is $\Gamma$-homomorphically equivalent to a monochromatic bipartite graph if and only if 
it is $\Gamma$-homomorphically equivalent to a monochromatic $K_1$ or $K_2$. 
The case for $K_1$ is trivial.
Consider the case where $H$ is 
$\Gamma$-homomorphically equivalent to a monochromatic $K_2$.
Without loss of generality, $H = K_2^i$.  Then $H^{\Gamma} = K_2^{\Delta(i)}$.
By Corollary~\ref{cor:switchhomAbel},
$\swhom{G}{H}$ if and only if $\swhom[{\Gamma^{ab}}]{G^{\Gamma}}{H^\Gamma}$, which can be tested
in polynomial time by 
Theorem \ref{thm:Arnott}.

Conversely if $H$ is not $\Gamma$-homomorphically equivalent to a monochromatic bipartite graph, then $H^{\Gamma}$ contains an odd cycle, or is bipartite but with a cycle that switches to be nearly monochromatic of some colours $(\Delta(i),\Delta(j))$.  Hence,
\swhhom[{$\Gamma^{ab}$}]{$H^{\Gamma}$} is NP-complete by Theorem~\ref{thm:abelianP}.  
Thus, \swhhom{$H$} is NP-complete by Corollary~\ref{cor:switchhomAbel}.
\end{proof}

The following corollary gives the structural characterization of the dichotomy.

\begin{corollary}
Let $\Gamma$ be an $m$-switching group.  For a fixed $m$-edge-coloured graph $H$, if
\begin{list}{(\roman{enumi})}{\usecounter{enumi}}
  \setlength\itemsep{0pt}
  \item $H$ is bipartite; and
  \item can be switched to be monochromatic,
\end{list}
then \swhhom{$H$} is polynomial-time solvable.  
Otherwise
\begin{list}{(\roman{enumi})}{\usecounter{enumi}}
  \setcounter{enumi}{2}
  \setlength\itemsep{0pt}
  \item $H$ contains an odd cycle, or 
  \item $H^{\Gamma}$ contains an (even) length cycle that can be switched (by $\Gamma^{ab}$) to be nearly monochromatic of colours $(\Delta(k), \Delta(l))$,
\end{list}
and \swhhom{$H$} is NP-complete.
\end{corollary}

The importance of nearly monochromatic cycles was identified in~\cite{kid21} where the agreeance class of $i$ (later renamed the substitution class of $i$ in~\cite{bre_kid_mac21}) is defined to be the set of all colours $j$ for which a nearly monochromatic cycle of colours $(i,j)$ can be switched to be monochromatic of colour $i$.  For even cycles these classes turn out to be exactly the block system $D = \{\Delta(1), \Delta(2), \ldots, \Delta(m) \}$ as defined by the action of the commutator subgroup. 

\section{Conclusion}

We provided a dichotomy for the \swhhom{$H$} problem for $m$-edge-coloured graphs and transitive groups $\Gamma$.  Natural next steps would be relax either of these two conditions and consider $n$-arc coloured graphs, $(m, n)$-mixed graphs, or groups that do not act transitively on the colours.
Later, one could relax the condition that the graphs be simple, and consider $(m, n)$-mixed graphs with loops, multiple edges / arcs of different colours, or directed 2-cycles.  Care must be taken when there are arcs because switching on the head then the tail of an arc may give a different result than switching in the reverse order.  One could first consider groups for which this difficulty does not arise (subgroups of $S_m \times (S_2 \wr S_n)$), and subsequently consider all possible groups.

\bigskip
\noindent\textbf{Acknowledgements\hspace{0.1in}} 

\smallskip\noindent
We acknowledge the support of the Natural Sciences and Engineering Research Council of Canada (NSERC). RGPIN-2014-04760 (Brewster) and RGPIN-2017-04459 (MacGillivray).

\smallskip\noindent
Nous remercions le Conseil de recherches en sciences naturelles et en g\'{e}nie du Canada (CRSNG) de son soutien. RGPIN-2014-04760 (Brewster) and RGPIN-2017-04459 (MacGillivray).

	\bibliography{bib}{}
	\bibliographystyle{plain}

\end{document}